\documentclass[12pt]{article}
\usepackage[cp1251]{inputenc}
\usepackage{amssymb,amsmath}
\usepackage[dvips]{graphicx,color}
\usepackage{mathrsfs}
\usepackage{array, amsfonts, mathrsfs}
\usepackage{amssymb}
\usepackage{amsthm}
\usepackage{graphics, graphicx}
\usepackage{epstopdf}
\usepackage[colorlinks=true]{hyperref}
\usepackage{float}
\usepackage{enumitem}

\newtheorem{Lemma}{Lemma}
\newtheorem{Theorem}{Theorem}

\newtheorem{Remark}{Remark}

\newtheorem{Proposition}{Proposition}

\title{Three-dimensional inverse acoustic scattering problem by the BC-method}
\author{M.I.Belishev\thanks {St.Petersburg Department of Steklov Mathematical
        Institute, St.Petersburg, Russia, e-mail: belishev@pdmi.ras.ru}\,\,\,\,and\,\,\,\,A.F.Vakulenko\thanks {St.Petersburg Department of Steklov Mathematical
        Institute, St.Petersburg, Russia, e-mail: vak@pdmi.ras.ru.}}
\date{}

\begin{document}

\maketitle

\begin{abstract}
Let $\Sigma:=[0,\infty)\times S^2$, $\mathscr F:=L_2(\Sigma)$. The
{\it forward} acoustic scattering problem under consideration is
to find $u=u^f(x,t)$ satisfying
\begin{align}
\label{Eq 01} &u_{tt}-\Delta u+qu=0, && (x,t) \in {\mathbb R}^3 \times (-\infty,\infty); \\
\label{Eq 02} &u \mid_{|x|<-t} =0 , && t<0;\\
\label{Eq 03} &\lim_{s \to -\infty}
s\,u((-s+\tau)\,\omega,s)=f(\tau,\omega), && (\tau,\omega)   \in
\Sigma;
\end{align}
for a real valued compactly supported potential $q\in
L_\infty(\Bbb R^3)$ and a control $f \in\mathscr F$. The response
operator $R: \mathscr F\to\mathscr F$,
\begin{align*}
& (Rf)(\tau ,\omega )\,:= \lim_{s \to +\infty} s\, u^f((s+\tau
)\,\omega ,s), \quad (\tau ,\omega ) \in \Sigma
\end{align*}
depends on $q$ {\it locally}: if $\xi>0$ and $f\in\mathscr
F^\xi:=\{f\in\mathscr F\,|\,\,\,f\!\mid_{[0,\xi)}=0\}$ holds, then
the values $(Rf)\!\mid_{\tau\geqslant\xi}$ are determined by
$q\!\mid_{|x|\geqslant\xi}$ (do not depend on
$q\!\mid_{|x|<\xi}$).

The {\it inverse problem} is: for an arbitrarily fixed $\xi>0$, to
determine $q\mid_{|x|\geqslant\xi}$ from $X^\xi
R\upharpoonright\mathscr F^\xi$, where $X^\xi$ is the projection
in $\mathscr F$ onto $\mathscr F^\xi$. It is solved by a relevant
version of the boundary control method. The key point of the
approach are recent results on the controllability of the system
(\ref{Eq 01})--(\ref{Eq 03}).
\end{abstract}

\section{Introduction}

\subsubsection*{About the paper}

The boundary control (BC-) method is an approach to inverse
problems of mathematical physics  \cite{B IP 97,B UMN,B IPI
2022,BKor,BSim_AA_2}. It is of interdisciplinary character and
exploits the various connections between the inverse problems and
control and system theory, function analysis (operator model
theory, Banach algebras), asymptotic methods (Geometrical Optics),
complex analysis. The given paper provides a version of the
BC-method relevant to the three-dimensional scattering problem for
the locally perturbed wave (acoustic) equation. Its peculiarity
and advantage is a local (time-optimal) recovering the parameters
under determination. This property is in demand in actual
applications, such as acoustics and geophysics.

\subsubsection*{Setups and results}

\noindent$\bullet$\,\,\ We denote  $B_R(x) :=\left \{x' \in {\Bbb
R}^3 \mid |x-x'| < R\right \}, S^2:=\left \{\theta \in {\Bbb R}^3
\mid |\theta|=1 \right \}$, $\Sigma:=[0,\infty)\times S^2$. Let $q
\in L_\infty(\Bbb R^3)$ be a real valued compactly supported
function ({\it potential}) provided ${\rm supp\,} q\subset
B_a(0)$. We assume $B_a(0)$ to be the minimal ball that contains
${\rm supp\,} q$ and say $a$ to be the radius of the potential.

The system under consideration is
\begin{align}
\label{Eq 1} & u_{tt}-\Delta u+qu=0, && (x,t) \in {\Bbb R}^3 \times (-\infty,0); \\
\label{Eq 2} & u \mid_{|x|<-t} =0 , && t<0;\\
\label{Eq 3} & \lim_{s \to -\infty}
s\,u((-s+\tau)\,\omega,s)=f(\tau,\omega), && (\tau,\omega) \in
\Sigma;
\end{align}
where $f\in L_2(\Sigma)$ is a {\it control}, $u=u^f(x,t)$ is a
solution ({\it wave}); the value $t=0$ is regarded as a final
moment. In the mean time, by the hyperbolicity of the system, its
extension
\begin{align}
\label{Eq 1ext} & u_{tt}-\Delta u+qu=0 \qquad \text{in}\,\,\,\, \{(x,t)\,|\,\,x\in\Bbb R^3,\,\,-\infty<t<|x|\}; \\
\label{Eq 2ext} & u \mid_{|x|<-t} =0 , \qquad \qquad \,\,\,t<0;\\
\label{Eq 3ext} & \lim_{s \to -\infty}
s\,u((-s+\tau)\,\omega,s)=f(\tau,\omega), \qquad (\tau,\omega) \in
\Sigma;
\end{align}
is well defined. With the extended system one associates a {\it
response operator}
\begin{align*}
& (Rf)(\tau ,\omega )\,:= \lim_{s \to +\infty} s\, u^f((s+\tau
)\,\omega ,s), \quad (\tau ,\omega ) \in \Sigma.
\end{align*}
The inverse problem, the setup of which is specified below, is to
determine $q$ from the given $R$.
\smallskip

\noindent$\bullet$\,\,\ To formulate the result, introduce the
spaces $\mathscr F:=L_2(\Sigma)$ and $\mathscr H:=L_2(\Bbb R^3)$
along with the families of their subspaces $\mathscr
F^\xi:=\{f\in\mathscr F\,|\,\,f\!\mid_{\tau<\xi}=0\}$ and
$\mathscr H^\xi:=\{y\in\mathscr H\,|\,\,\,y\mid_{B_\xi(0)}=0\}$
for $\xi>0$. Let $X^\xi$ be the projection in $\mathscr F$ onto
$\mathscr F^\xi$ which cuts off controls on
$\Sigma^\xi:=\{(\tau,\omega)\in\Sigma\,|\,\,\tau\geqslant\xi\}$.

As is known, the operator $W:\mathscr F\to\mathscr H$,
$Wf:=u^f(\cdot,0)$ is bounded (see, e.g., \cite{BV_3}). By
hyperbolicity of the system, $f\in\mathscr F^\xi$ implies
$u^f(\cdot,0)\in \mathscr H^\xi$, whereas the values of
$u^f(\cdot,0)$ are determined by the part $q\!\mid_{|x|\geqslant
\xi}$ of the potential (do not depend on $q\!\mid_{|x|< \xi}$). By
the same arguments, the values $(Rf)\!\mid_{\tau\geqslant\xi}$ are
also determined by $q\!\mid_{|x|\geqslant\xi}$. Such a locality of
the correspondence $q\mapsto R$ motivates the following setup of
the {\it inverse problem}: for an arbitrarily fixed $\xi>0$, to
determine $q\mid_{|x|\geqslant\xi}$ from the operator $X^\xi
R\upharpoonright\mathscr F^\xi$, which acts in the subspace of the
"delayed"\, controls $\mathscr F^\xi$.
\smallskip

Our main result is the following. For an arbitrary $\xi>0$, given
the operator $R^\xi:=X^\xi R\upharpoonright\mathscr F^\xi$, which
acts in $\mathscr F^\xi$, we recover the operator $W^\xi:=
W\upharpoonright\mathscr F^\xi$, which acts from $\mathscr F^\xi$
to $\mathscr H^\xi$. The corresponding procedure is referred to as
a {\it wave visualization}. The knowledge of $W^\xi$ enables one
to recover the graph of the operator
$(\Delta-q)\upharpoonright\mathscr H^\xi$. The graph evidently
determines the part $q\!\mid_{\Bbb R^3\setminus B_\xi(0)}$ of the
potential. A locality (time-optimality) of the determination
$R^\xi\Rightarrow q\!\mid_{\Bbb R^3\setminus B_\xi(0)}$ is a
specifics and main advantage of the BC-method.

\subsubsection*{Comments}

\noindent$\bullet$\,\,\ The problem under consideration is an "old
debt"\, of the BC-method. Its solution is prepared by the papers
\cite{BV_2,BV_3,BV JIIPP centrifug pot,BV s-points JMA 2010}. The
final step was made in the recent paper \cite{BV_11}, which has
revealed the character of controllability of the system (\ref{Eq
1})--(\ref{Eq 3}). This is in keeping with the system theory
thesis, which the BC-method follows to: the better a system is
controllable, the richer the information, which can be extracted
from external observations on the system \cite{KFA}. Fortunately,
in the given case, the controllability turns out to be suitable
for applying the BC-technique \cite{BV_11}.
\smallskip

\noindent$\bullet$\,\,\ One more philosophical thesis coming from
the system theory \cite{KFA}, which the BC-method follows to, is
the following. The observer, which implements external
measurements on a system, is not able to operate with its
(invisible) inner states. To recover the system, the observer
constructs certain "copies" of the inner states, extracting them
from the measurement data. Such a constructing is interpreted as a
{\it visualization}. In all versions of the BC-method such copies
are present and used in explicit or implicit form. In the given
paper, this role is played by the space $\tilde{\mathscr H^\xi}$,
which provides the copies $\tilde u^f$ of the (invisible) waves
$u^f$ in $\mathscr H^\xi$. The space $\tilde{\mathscr H^\xi}$ is
constructed via the response operator $R^\xi$.

Note that visualization of waves through time-domain inverse data
has been used in other works as well: see, e.g., \cite{R_Ch_DeF}.
\smallskip

\noindent$\bullet$\,\,\ The main technical tool for solving a
class of inverse problems by the BC-method is the {\it amplitude
integral} (AI), which is a generalization of the triangular
truncation operator integral by M.S.Brodskii and M.G.Krein
\cite{Brod,GK,BKach,B IPI 2022}. We provide its version relevant
for the scattering problem. The theoretical-operator scheme of the
BC-method based upon the triangular factorization, is presented in
the paper \cite{BSim_AA_2}, where the AI is referred to as a {\it
diagonal} of the operator $W$.

\section{Forward problem}\label{Sec Forward}

\subsubsection*{Dynamical system: spaces and operators}

The following is the standard system theory attributes of the
system (\ref{Eq 1})--(\ref{Eq 3}).

\noindent$\bullet$\,\,\ The {\it outer space} of controls is
${\mathscr F} :=L_2(\Sigma)$. It contains the subspaces
\begin{equation*}
{\mathscr F}^\xi := \left \{ f \in {\mathscr
F}\,|\,\,\,\,f\!\mid_{\tau<\xi}=0 \right \}, \quad \xi > 0
\end{equation*}
consisting of the delayed controls \,($\xi$ is the delay).
\smallskip

\noindent$\bullet$\,\,\ The {\it inner space} of states is
${\mathscr H} := L_2({\Bbb R}^3)$; the waves $u^f(\cdot,t)$ are
the time dependent elements of $ {\mathscr H}$. It contains the
subspaces
\begin{equation*}
  {\mathscr H}^\xi := \left \{ y \in {\mathscr H}\!\mid
{\rm supp}\, y \subset    {\Bbb R}^3 \setminus B_\xi(0)  \right
\}, \quad \xi > 0.
\end{equation*}

Also, the inner space contains the {\it reachable sets}
\begin{equation*}
 {\mathscr U}^\xi \,:=\,\left \{  u^f(\cdot,0)\!\mid
f \in {\mathscr F}^\xi  \right \},  \quad \xi > 0.
\end{equation*}
They are the closed subspaces and, by hyperbolicity of the problem
(\ref{Eq 1})--(\ref{Eq 3}), the embedding ${\mathscr U}^\xi
\subset {\mathscr H}^\xi$ holds (see, e.g., \cite{BV_3}).

The subspaces (unreachable sets)
\begin{equation*}
{\mathscr D}^\xi :=  {\mathscr H}^\xi \ominus
{\mathscr U}^\xi , \quad \xi > 0
\end{equation*}
are called {\it defect} subspaces. The following important fact is
recently established in \cite{BV_11}. We say a $y\in\mathscr
H^\xi$ to be a {\it polyharmonic} function of the order $n\in\Bbb
N$  if $(-\Delta +q)^n\, y=0$ holds in ${\Bbb R}^3 \setminus
\overline{B_\xi(0)}$, and write $y\in\mathscr A^\xi_n$. Denote
$\mathscr A^\xi:={\overline{{\rm span\,}\{\mathscr
A^\xi_n\,|\,\,n\geqslant 1\}}}$.
\begin{Proposition}\label{Prop H=U+D}
The relation
\begin{equation}\label{Eq char Def}
\mathscr A^\xi\,=\,{\mathscr D}^\xi,\qquad\xi>0
\end{equation}
(the closure in $\mathscr H$) holds.
\end{Proposition}
The relation (\ref{Eq char Def}) enhances the embedding $\mathscr
A^\xi\subset \mathscr D^\xi$ proved in \cite{BV_3}
\smallskip

\noindent$\bullet$\,\,\ The {\it control operator} of the system
is $W: {\mathscr F} \to {\mathscr H}$ \quad
$$
Wf\,:=\,u^f(\cdot,0).
$$
It is bounded \cite{LP,BV_3} and the representation
\begin{equation*}
W\,=\,W_0+K
\end{equation*}
holds with a compact operator $K$, where $W_0$ is the control
operator of the (unperturbed) system (\ref{Eq 1})--(\ref{Eq 3})
with $q=0$. Note that the compactness of $K$ is proved in
\cite{BV_3} under the assumption that the potential $q$ is
compactly supported. In what follows, by $u^f_0$ we denote the
waves in the unperturbed system.

Recall that ${\rm supp\,}q\subset B_a(0)$.  The influence domain
of the potential is
$$
D\,:=\,\{(x,t)\,|\,\,t>-a,\,\,t<|x|<t+2a\}.
$$
Outside it, the perturbed and unperturbed waves coincide: we have
\begin{equation}\label{Eq u=u0 outside D}
u^f\,=\,u^f_0 \qquad {\text{in}}\,\,\,\Bbb R^3\setminus D.
\end{equation}

Operator $W_0$ is unitary: $W_0^*=W_0^{-1}$ holds. Later on we
consider $W_0$ and $W$ in more detail. \smallskip

\noindent$\bullet$\,\,\ The {\it response operator} $R:\mathscr
F\to\mathscr F$
\begin{equation*}
(Rf)(\tau ,\omega ):= \lim_{s \to +\infty} s\, u^f((s+\tau
)\,\omega ,s), \quad (\tau ,\omega ) \in \Sigma
\end{equation*}
is associated with the extended system (\ref{Eq 1ext})--(\ref{Eq
3ext}). It is compact and self-adjoint \cite{BV_3}. The relation
(\ref{Eq u=u0 outside D}) easily leads to
\begin{equation}\label{Eq Rf2a=0}
Rf\!\mid_{\tau>2a}\,=\,0, \qquad f\in \mathscr F.
\end{equation}

Since the operator $\Delta-q$ that governs the evolution of the
system, does not depend on time, the relation
\begin{equation}\label{Eq Delay Rel}
u^f(\cdot, t-h)\,=\,u^{T_h f}(\cdot,t), \qquad -\infty<
t<\infty,\,\,h\geqslant 0
\end{equation}
holds, where $T_h$ is the delay (shift) operator acting in
$\mathscr F$ by $(T_hf)(\cdot,t):=f(\cdot,t-h)$ (assuming
$f\!\mid_{\tau<0}=0$). As a consequence, one can derive the
relation $RT_h=T_h^*R$.

If the potential is smooth enough, The response operator can be
represented in the form
\begin{equation}\label{Eq Repres R}
(Rf)(\tau,\omega)\,=\,\int_\Sigma
p\,(t+s;\,\omega,\theta)\,d\tau\,d\theta,\qquad
(\tau,\omega)\in\Sigma.
\end{equation}
The dependence of the kernel on the sum $t+s$ corresponds to the
intertwinning with the shift mentioned above. In the mean time, by
virtue of (\ref{Eq Rf2a=0}), the kernel $p$ obeys ${\rm
supp\,}p\subset [0,2a]\times S^2\times S^2$.
\smallskip

\noindent$\bullet$\,\,\ A map $ C:{\mathscr F} \to {\mathscr F}$,
$$
C\,:=\,W^*W
$$
is called a {\it connecting operator}. It is a bounded positive
operator. By the definition, for $f, g \in {\mathscr F}$ one has
\begin{equation}\label{Eq 2.31}
(Cf,g)_{{\mathscr F}}=(Wf,Wg)_{{\mathscr H}}= \left (
u^f(\cdot,0),u^g(\cdot,0)\right )_{\mathscr H}~,
\end{equation}
i.e, $C$ connects the Hilbert metrics of the outer and inner
spaces. As is shown in \cite{BV_3}, the relation
\begin{equation}\label{Eq C=I+R}
C=I+R
\end{equation}
holds.

\subsubsection*{Unperturbed system}

The unperturbed system is
\begin{align}
\label{Eq 10} & u_{tt}-\Delta u=0, && (x,t) \in {\mathbb R}^3 \times (-\infty,0); \\
\label{Eq 20} & u \!\mid_{|x|<-t} =0 , && t<0;\\
\label{Eq 30} & \lim_{s \to -\infty}
s\,u((-s+\tau)\,\omega,s)=f(\tau,\omega), && (\tau,\omega)   \in
\Sigma;
\end{align}
the waves are $u^f_0(x,t)$. Here are some known facts about it
taken from the papers \cite{LP,BV_2,BV_3}.
\smallskip

\noindent$\bullet$\,\,\, The solution $u_0^f$ can be represented
in explicit form as follows.

Fix $\omega \in S^2$ and define
\begin{equation*}
 \pi_b(\omega):=
 \begin{cases}
\left \{     \theta \in S^2
     \!\mid \omega \cdot \theta =b
\right \},  \quad &b \in [-1,1];\\
\emptyset ,     \quad & |b| >1;
    \end  {cases} .
\end{equation*}
The set $ \pi_b(\omega)$ is a parallel on the unit sphere with the
North Pole $\omega,$ the length of the parallel is equal to $2 \pi
\sqrt{1-b^2};$  $ \pi_0(\omega)$ is the equator; $ \pi_{\pm
1}(\omega)= \pm \omega$. For a function $g$ on $S^2$, denote by
\begin{equation}\label{Eq Funk Mink}
 [g]_b(\omega):=
 \begin{cases}
   \frac{1}{2\pi \sqrt{1-b^2}}
   \int\limits_{\pi_b(\omega)} g(\theta)\, d\theta ,
 \quad
 & b \in (-1,1);\\
 g(-\omega) ,
  &b= -1;\\
 g(\omega) ,
  &b= 1;\\
  0,   &|b|>1;
    \end  {cases}
    \end{equation}
the mean value of $g$ on the parallel. The following result is
established in \cite{BV_3}.
\begin{Lemma}
Let a control $f$ and its derivative $f_\tau$ belong to $\mathscr
F$. Then the representation
\begin{equation}\label{Eq repres u^f_0}
u^f_0(x,t)=
 \frac{1}{2 \pi}
\int\limits_{ S^2} f_{\tau}  (t+r\,\omega \cdot \theta,\theta)
 \, d\theta
 + \frac{1}{r}\, [f(0,\cdot)]_{-  \frac{t}{r}}\, (\omega),\qquad
 x\in\Bbb R^3,\,\,\,t\leqslant 0
 \end{equation}
holds, where $r=|x|$,\,\, $\omega=\frac{x}{|x|}$;\,\, $a\cdot b$
is the standard inner product in $\Bbb R^3$, and $f_\tau$ is
extended to $\tau <0$ by zero.
\end{Lemma}

Note a peculiarity of this representation: the summands in
(\ref{Eq repres u^f_0}) may be not square integrable in $\Bbb R^3$
but the sum does belong to $\mathscr H$  \cite{BV_3}.  At the same
time, both summands are the solutions of the wave equation
(\ref{Eq 10}).
\medskip

\noindent$\bullet$\,\,\, An important fact used in what follows is
that the hyperbolic problem (\ref{Eq 10})--(\ref{Eq 30}) is well
posed for {\it any} control provided $f, f_\tau\in L_2^{\rm
loc}(\Sigma)$, regardless of its behavior at $\tau\to\infty$,
whereas the corresponding solution $u^f_0(\cdot,t)\in L_2^{\rm
loc}(\Bbb R^3),\,t\leqslant 0$ is given by the same formula
(\ref{Eq repres u^f_0}). We say such $f$'s to be {\it admissible}
and, if otherwise is not specified, deal with controls of this
class.

In particular, the solution $u^f_0$ is well defined for the {\it
polynomial controls}
\begin{align*}
\notag & \mathscr P\,:=\,\\
&
\left\{p^l_{jm}(\tau,\omega)=\tau^{l-2j}\,Y_l^m(\omega)\,\bigg|\,\,
l=0,1,\dots;\,\,0\leqslant j\leqslant
\left[\frac{l}{2}\right];\,\, -l\leqslant m\leqslant l\right\},
\end{align*}
where $Y^m_l$ are the standard spherical harmonics, $[...]$ is the
integer part. In contrast to them, we say the controls belonging
to $\mathscr F$ to be {\it ordinary}. In what follows we operate
with controls of the class $\mathscr F\dotplus\mathscr P$.

There are two properties that distinguish the class $\mathscr P$.
First, for $p\in\mathscr P$ the waves $u_0^p$ are expressed via
controls $p$ in explicit form \cite{BV_2,BV JIIPP centrifug pot}.
Second, these waves vanish at $t=0$ and are odd w.r.t. time: the
relations
\begin{equation}\label{Eq u^p(0)=0}
u^p_0(\cdot,0)\,=\,0,\quad u^p_0(\cdot,t)=-u^f_0(\cdot,-t),\qquad
p\in\mathscr P
\end{equation}
hold \cite{BV_2,BV_3}. Note that since $W_0$ is a unitary operator
from $\mathscr F$ to $\mathscr H$, for the ordinary controls $f$
the relations (\ref{Eq u^p(0)=0}) are impossible.
\begin{Remark}
Note in advance that, owing to the property (\ref{Eq u=u0 outside
D}), the perturbed solutions $u^f$ are also well defined for
controls $f\in\mathscr F\dotplus\mathscr P$. This fact is
substantially used in the proof of Lemma \ref{L perturbed C-form}.
\end{Remark}
\medskip

\noindent$\bullet$\,\,\ The well-known fact of the hyperbolic PDE
theory is that the singular controls initiate the singular waves,
with the singularities propagate along the characteristics. The
relations, which express singularities of waves via singularities
of controls, are usually called the geometrical optics (GO)
formulas. The following result is of this kind. We denote
$S^2_R:=\{x\in \Bbb R^3\,|\,\,|x|=R\}$ and recall that
$B_R(x)=\{y\in \Bbb R^3\,|\,\,|x-y|<R\}$.
\smallskip

Take an admissible control $f$ provided $f(0,\cdot)\not=0$. So,
being extended to $\tau<0$ by zero, $f$ has a break of its
amplitude at $\tau=0$. As a consequence, the wave $u^f_0$, which
is supported in the domain $\{(x,t)\,|\,\,|x|\geqslant -t\}$,
turns out to be discontinuous near the characteristic cone
$\{(x,t)\,|\,\,t< 0,\,\,|x|= -t\}$. In other words, for any $t<0$
the wave $u_0^f(\cdot,t)$ has a break of amplitude at its forward
front $S^2_{-t}$. The values of the breaks are related as follows.
Denote
$$
s^p_+\,:=\,\begin{cases} 0, & s<0;\\
s^p, &s\geqslant 0;
\end{cases};\qquad p\geqslant 0,
$$
so that $s^0_+$ is the Heaviside function.
\begin{Lemma}\label{L ampl breaks unperturb}
Let $f\in C^2_{\rm loc}(\Sigma)$; fix a $t<0$ and a (small)
$\delta>0$. The GO-representation
\begin{equation}\label{Eq GO 0}
u^f_0(x,t)\,=\,\frac{f(0,\omega)}{r}\,(r+t)^0_+
\,+\,w_0(x,t,\delta)\,(r+t)^1_+,\qquad 0\leqslant
r+t\leqslant\delta
\end{equation}
holds, where $r=|x|$, $\omega=\frac{x}{|x|}$, and the estimate
$|w_0|\leqslant c_0\,\|f\|_{C^2([0,\delta]\times S^2)}$ is valid
uniformly w.r.t. $x$ and $t$.
\end{Lemma}
\begin{proof}
\noindent{\bf 1.}\,\,\,Take a $g\in C^2(S^2)$ and $b=1-\delta$
with a small $\delta>0$. Representing by Tailor-Lagrange
\begin{equation}\label{Eg Tailor}
g(\theta)\,=\,g(\omega)+\nabla_\theta
\,g\,(\omega)\cdot(\theta-\omega)+\left(B(\omega,\theta)(\theta-\omega)\right)\cdot(\theta-\omega)
\end{equation}
(here $\theta$ and $\omega$ are considered as vectors in $\Bbb
R^3\supset S^2$) and integrating over the (small) parallel
$\pi_{1-\delta}(\omega)$, in accordance with (\ref{Eq Funk Mink}),
we get
\begin{equation}\label{Eq g+tilde h delta}
[g]_{1-\delta}(\omega)=
\frac{1}{|\pi_{1-\delta}(\omega)|}\,\int_{\pi_{1-\delta}(\omega)}g(\theta)\,d\theta\,\,\overset{(\ref{Eg
Tailor})}=\,\,g(\omega)+h(\omega,\delta)\,\delta
\end{equation}
with $h$ obeying $|h|\leqslant c\,\|g\|_{C^2(S^2)}$. Note that the
first-order term with $\nabla_\theta\,g$ vanishes in course of
integration over the parallel.
\smallskip

\noindent{\bf 2.}\,\,\,Applying (\ref{Eq g+tilde h delta}) to the
second summand in (\ref{Eq repres u^f_0}), we have
\begin{align}\label{Eg second summand}
\notag & \frac{1}{r}\, [f(0,\cdot)]_{-\frac{t}{r}}\,
(\omega)=\frac{1}{r}\, [f(0,\cdot)]_{1-\frac{r+t}{r}}\,
(\omega)=\frac{f(0,\omega)+ h(r,\omega,t,\delta)(r+t)}
{r}=\\
& =\frac{f(0,\omega)}{r}\,(r+t)^0_+ +
w_1(r,t,\omega,\delta)\,(r+t)^1_+,\qquad 0\leqslant
r+t\leqslant\delta
\end{align}
with $|w_1|\leqslant c_1\,\|f(0,\cdot)\|_{C^2(S^2)}$.
\smallskip

\noindent{\bf 3.}\,\,\,As it easily follows from (\ref{Eq repres
u^f_0}), the values $u^f_0\!\mid_{0\leqslant r+t\leqslant\delta}$
are determined by the values $f\!\mid_{0\leqslant
\tau\leqslant\delta}$ (does not depend on $f\!\mid_{
\tau>\delta}$). Estimating the first summand in (\ref{Eq repres
u^f_0}) for $0\leqslant r+t\leqslant\delta$, one has
\begin{align}
\notag &
\bigg|\int_{S^2}f_\tau(t+r\omega\cdot\theta,\theta)\,d\theta\bigg|\leqslant
\|f\|_{C^1([0,\delta]\times S^2)}\,{\rm mes\,}\{\theta\in
S^2\,|\,\,t+r\omega\cdot\theta\geqslant 0\}=\\
\notag & = \|f\|_{C^1([0,\delta]\times S^2))}\,{\rm
mes\,}\left\{\theta\in S^2\,|\,\,\cos\theta\geqslant
1-\frac{r+t}{r}\right\}.
\end{align}
One can easily verify that the measure is an infinitesimal of the
order $r+t$, which implies
\begin{align}\label{Eg first summand}
\frac{1}{2\pi}\,\int_{S^2}f_\tau(t+r\omega\cdot\theta,\theta)\,d\theta
=\,w_2(r,t,\omega,\delta)\,(r+t)^1_+,\qquad 0\leqslant
r+t\leqslant\delta,
\end{align}
with $|w_2|\leqslant c_2\,\|f\|_{C^1([0,\delta]\times S^2)}$.
\smallskip

\noindent{\bf 4.}\,\,\,Joining (\ref{Eg second summand}) and
(\ref{Eg first summand}), we arrive at (\ref{Eq GO 0}).
\end{proof}

\subsubsection*{Perturbed system}

\noindent$\bullet$\,\,\ Return to the system (\ref{Eq
1ext})--(\ref{Eq 3ext}) and recall that $q\in L_\infty(\Bbb R^3)$
and ${\rm supp\,}q\subset B_a(0)$ holds. Recall the coincidence of
solutions (\ref{Eq u=u0 outside D}) outside the influence domain
$D$. In particular, $u^f=u^f_0$ holds for $t\leqslant -a$. The
latter enables to present (\ref{Eq 1ext})--(\ref{Eq 3ext}) in the
equivalent problem
\begin{align*}
& (u-u_0)_{tt}-\Delta(u-u_0)\,=\,-\,qu && {\rm in}\,\,\Bbb
R^3\times (-a,0);\\
& (u-u_0)\!\mid_{t<-a}=0,
\end{align*}
and then reduce it to the integral equation by the Kirchhoff
formula
\begin{align}
\notag &
u^f(x,t)\,=\,u^f_0(x,t)\,-\,\frac{1}{4\pi}\,\int_{B_{t+a}(x)}\frac{q(y)\,u^f(y,
t+a-|x-y|)}{|x-y|}\,dy=\\
\label{Eq Kirchhoff} & =\,u^f_0(x,t)\,-\,(I
u^f)(x,t),\qquad(x,t)\in D.
\end{align}
This equation, in turn, is reduced to the family of the equations
of the same form in the spaces $L_2(D^b)$, where $D^b:=\{(x,t)\in
D\,|\,\, t\leqslant b\},\,\,a< b<\infty$.

Fix a \,\,\,$b>a$. Assuming $q\in L_\infty(\Bbb R^3)$, operator
$I$ acts in $L_2(D^b)$, is compact (see \cite{BV_3}, (2.4)), and
possesses a continuous nest of the invariant subspaces of
functions supported in $D^\eta$ with $\eta\leqslant b$. As such,
$I$ is a Volterra operator \cite{Brod}. Hence, the operator
$\mathbb I-I$ is boundedly invertible in each $L_2(D^b)$ and we
have $u^f=(\Bbb I-I)^{-1}u^f_0\in L_2(D^b)$. For a locally
square-summable solution $u^f$, the integral $Iu^f$ depends on
$(x,t)\in D$ continuously. Therefore,  if $f,f_\tau \in C_{\rm
loc}(\Sigma)$, then $u^f_0\in C_(D^b)$ holds by (\ref{Eq repres
u^f_0}). Hence, $u^f=u^f_0-Iu^f\in C(D^b)$ holds.

By arbitrariness of $b$, we arrive at $u^f\in C_{\rm loc}(D)$.
\medskip

\noindent$\bullet$\,\,\, The Kirchhoff formula is a relevant tool
for the GO-analysis. Treating the behavior of the wave $u^f$ near
its forward front, we already have the GO-representation (\ref{Eq
GO 0}) for the summand $u^f_0$ in (\ref{Eq Kirchhoff}) and, thus,
it remains to estimate the contribution of the second summand. Let
us do it.
\smallskip

Recall the assumption $f\in C^2_{\rm loc}(\Sigma)$, which provides
the continuity of $u^f_0$ and then the continuity of $u^f$. As
above, we denote $r=|x|,\,\,\omega=\frac{x}{|x|}$.

Fix a $t<0$. By (\ref{Eq Kirchhoff}), the integral $Iu^f$ is taken
over the (3-dimensional) cone $C_{x,t}:=\{(y,s)\,|\,\,-a\leqslant
s\leqslant t,\,\,|x-y|=t-s\}$. In the mean time, $u^f$ vanishes
for $r<-t$. Therefore, in fact the integral is taken over the part
$\dot C_{x,t}:=C_{x,t}\cap\{(y,s)\,|\,\,|y|\geqslant-t\}$. When
the top $(r\omega,t)$ of $C_{x,t}$ approaches to the point
$(|t|\,\omega,t)$, which lies at the forward front of
$u^f(\cdot,t)$ (i.e., when $r+t\to 0$), this part shrinks to the
segment $l_{x,t}$ of the straight line that connects the point
$(|t|\,\omega,t)$ with the point $(a\,\omega,-a)$ in $\Bbb R^4$.

Then we represent $\dot C_{x,t}=\dot C'_{x,t}\cup\dot C''_{x,t}$,
where $\dot C'_{x,t}:=\{(x',t')\in \dot
C_{x,t}\,|\,\,t-(r+t)\leqslant t'\leqslant t\}$ is the (small)
cone of the height $\frac{r+t}{2}$, and $\dot C''_{x,t}=\dot
C_{x,t}\setminus \dot C'_{x,t}$ is a rest. The part $\dot
C'_{x,t}$ contains the top $(x,t)$, where the integrant of $Iu^f$
has a singularity $\frac{1}{|x-y|}$. As is easy to check,
$\int_{\dot C'_{x,t}}$ is of the order $r+t$.

In the mean time, the part $\dot C''_{x,t}$ is projected along the
generating straight lines of the cone $C_{x,t}$ onto the domain
$B_{t+a}(x)\setminus B_a(0)\subset \Bbb R^3$, which is of the
transversal size $|\,r\omega-|t|\,\omega\,|=r+t\to 0$. By the
latter, we have ${\rm mes\,}[B_{t+a}(x)\setminus B_a(0)]\sim
r+t\to 0$. Thus, the measure of $\dot C''_{x,t}$ is an
infinitesimal of the order $r+t$. As a result, the integral
$\int_{\dot C''_{x,t}}$ vanishes as $r+t$.

Summarizing the above considerations, we obtain the representation
\begin{equation}\label{Eq estim I}
(Iu^f)(x,t)\,=\,\dot w(x,t,\delta)\,(r+t)^1_+,\qquad 0\leqslant
r+t\leqslant \delta,
\end{equation}
where $\dot w$ obeys
\begin{align*}
&|\dot w|\leqslant c_1\,\|q\|_{L_\infty(\Bbb
R^3)}\,\|u^f\|_{C(D^a)}\leqslant
c_2 \,\|q\|_{L_\infty(\Bbb R^3)}\,\|u^f_0\|_{C(D^a)}\leqslant\\
& \leqslant\,c_3\,\|q\|_{L_\infty(\Bbb
R^3)}\|f\|_{C^1([0,\delta]\times S^2)}
\end{align*}
with the relevant constants. The latter inequality follows from
the fact that $u^f_0\!\mid_{0\leqslant r+t\leqslant\delta}$ is
determined by $f\!\mid_{0\leqslant \tau\leqslant\delta}$.

Note that, analyzing the shrinking of $\dot C_{x,t}\to l_{x,t}$ in
more detail, under additional assumptions on the smoothness of
$q$, one can derive more precise classical GO-formulas (see, e.g.,
Appendix in \cite{BKach}).
\smallskip

\noindent$\bullet$\,\,\ Joining (\ref{Eq GO 0}) with (\ref{Eq
Kirchhoff}) and (\ref{Eq estim I}), and denoting $w:=w_0-\dot w$,
we arrive at the following GO-representation for the perturbed
waves.
\begin{Lemma}\label{L ampl breaks perturb}
Let $f\in C^2_{\rm loc}(\Sigma)$; fix a $t<0$ and a (small)
$\delta>0$. The representation
\begin{equation}\label{Eq GO perturb}
u^f(x,t)\,=\,\frac{f(0,\omega)}{r}\,(r+t)^0_+
\,+\,w(x,t,\delta)\,(r+t)^1_+,\qquad 0\leqslant r+t\leqslant\delta
\end{equation}
holds, where $r=|x|$, $\omega=\frac{x}{|x|}$, and the estimate
$|w|\leqslant c\,\|q\|_{L_\infty(\Bbb
R^3)}\,\|f\|_{C^2([0,\delta]\times S^2)}$ is valid uniformly
w.r.t. $x$ and $t$.
\end{Lemma}
Thus, if a control $f$ has an onset $f(0,\cdot)\not=0$ at $\tau=0$
then the wave $u^f$ has a jump on its forward front, the amplitude
of the jump being equal to $\frac{f(0,\cdot)}{r}$. The amplitude
grows, when $r\to 0$ that corresponds the focusing effect.
\smallskip

The coincidence of the forms of (\ref{Eq GO 0}) and (\ref{Eq GO
perturb}) reflects the well-known fact: the presence of the
zero-order term $q$ in the operator $\Delta-q$ that governs the
evolution of the perturbed system, does not influent on the
leading terms in GO-formulas.

\section{Inverse problem}

\subsubsection*{Amplitude integral}

The {\it amplitude integral} (AI) is an operator construction,
which is a basic tool for solving the dynamical inverse problems
by the BC-method \cite{B IP 97,BKach,B IPI 2022,BSim_AA_2}. It is
a generalization of the classical triangular truncation integral
\cite{Brod,GK}. Here a version of AI relevant to the acoustic
scattering problem, is provided.
\smallskip

\noindent$\bullet$\,\,\  Fix a $\xi> 0$. Recall that $X^\xi$ is
the projection in $\mathscr F$ on $\mathscr F^\xi$ that cuts off
controls on $\Sigma^\xi$.  By $Y^\xi$ we denote the projection in
$\mathscr H$ on $\mathscr H^\xi$ that cuts off functions on the
exterior of the ball $B_\xi(0)$.

Let $\Pi^\delta:=\{\tau_k\}_{k\geqslant
0},\,\,\,0=\tau_0<\tau_1<\tau_2<\dots\to\infty$ be a partition of
the semi-axis $0\leqslant \tau<\infty$ provided
$0<\tau_k-\tau_{k-1}\leqslant \delta$. The difference $\Delta
X^{\tau_k}:=X^{\tau_k}-X^{\tau_{k-1}}$ projects in $\mathscr F$ on
$\mathscr F^{\tau_k}\ominus\mathscr F^{\tau_{k-1}}$ and thus cuts
off controls on $\Sigma^{\tau_k}\setminus\Sigma^{\tau_{k-1}}$. The
difference $\Delta Y^{\tau_k}:=Y^{\tau_{k-1}}-Y^{\tau_k}$ projects
in $\mathscr H$ on $\mathscr H^{\tau_{k-1}}\ominus\mathscr
H^{\tau_{k}}$, i.e., cuts off functions on the spherical layer
$B_{\tau_{k}}(0)\setminus B_{\tau_{k-1}}(0)$. Note the evident
orthogonality relations:
\begin{align}\label{Eq orthogonal}
\notag &\Delta X^{\tau_k}\Delta X^{\tau_l}=\Bbb O_\mathscr
F,\,\,\,\Delta
Y^{\tau_k}\Delta Y^{\tau_l}=\Bbb O_\mathscr H,\quad k\not= l;\\
& \Delta X^{\tau_k}X^{\tau_l}=\Bbb O_\mathscr F,\,\,\,\Delta
Y^{\tau_k}Y^{\tau_l}=\Bbb O_\mathscr H,\quad l\leqslant k-1.
\end{align}

Recall that $W$ is a control operator of the perturbed system
(\ref{Eq 1})--(\ref{Eq 3}) and introduce the sum
\begin{align}\label{Eq AI sums}
A_{\Pi^\delta}\,:=\,\sum\limits_{k\geqslant 1}\Delta
Y^{\tau_k}W\,\Delta X^{\tau_k},
\end{align}
which is an operator from $\mathscr F$ to $\mathscr H$, well
defined on the compactly supported controls. By the orthogonality
(\ref{Eq orthogonal}) we have
\begin{align*}
& \|A_{\Pi^\delta}f\|_\mathscr H^2=\left\|\sum\limits_{k\geqslant
1}\Delta Y^{\tau_k}W\,\Delta X^{\tau_k}f\right\|^2_\mathscr
H=\sum\limits_{k\geqslant 1}\|W\,\Delta
X^{\tau_k}f\|_\mathscr H^2\leqslant\\
& \leqslant \|W\|^2\sum\limits_{k\geqslant 1}\|\Delta
X^{\tau_k}f\|_\mathscr F^2= \|W\|^2\,\|f\|^2_\mathscr F,
\end{align*}
so that the sum (\ref{Eq AI sums}) obeys
$\|A_{\Pi^\delta}\|\leqslant\|W\|$ and, hence, is a bounded
operator.
\smallskip

\noindent$\bullet$\,\,\  As is easy to check, the map $A:\mathscr
F\to\mathscr H$,
\begin{equation*}
(Af)(x)\,:=\,\frac{f(r,\omega)}{r},\qquad x=r\omega\in\Bbb R^3
\end{equation*}
is a unitary operator. Its adjoint $A^*=A^{-1}:\mathscr
H\to\mathscr F$ is of the form
\begin{equation}\label{Eq A^*}
(A^*y)(\tau,\omega)\,:=\,\tau\,y(\tau\omega),\qquad (\tau,\omega)
\in \Sigma.
\end{equation}
\begin{Theorem}\label{T AI}
The relation
\begin{equation}\label{Eq AI}
A\,=\,s\text{--}\lim_{\delta\to 0}A_{\,\Pi^\delta}
\end{equation}
holds.
\end{Theorem}
\begin{proof}

\noindent{\bf 1.}\,\,\,Take an $f\in\mathscr F\cap C^2_{\rm
loc}(\Sigma)$. Fix a $\xi> 0$ and a (small) $\delta>0$. The cut
off control $X^\xi f$ has a break at $\tau=\xi$, its amplitude
(onset) being equal to $f(\xi,\cdot)$. Therefore, by (\ref{Eq GO
perturb}) and with regard to the delay relation (\ref{Eq Delay
Rel}), we have
\begin{equation*}
u^{X^\xi f}(x,t)=\frac{f(\xi,\omega)}{r}\,(r+t-\xi)^0_+
\,+\,w(x,t-\xi,\delta)\,(r+t-\xi)^1_+,\quad 0\leqslant
r+t-\xi\leqslant\delta.
\end{equation*}
Putting $t=0$ and representing $f(\xi,\omega)=f(r,\omega)+
w'(r,\xi)\,(r-\xi)$, we get
\begin{equation}\label{Eq auxil 2}
u^{X^\xi f}(x,0)=(Wf)(x)=\frac{f(r,\omega)}{r}\,
\,+\,w(x,-\xi,\delta)\,(r-\xi),\quad \xi\leqslant
r\leqslant\xi+\delta.
\end{equation}

\noindent{\bf 2.}\,\,\,For the above chosen $f$, by (\ref{Eq auxil
2}) the summands of $A_{\Pi^\delta}f$ are of the form
\begin{equation}\label{Eq auxil 3}
\left(\Delta Y^{\tau_k}W\,\Delta
X^{\tau_k}f\right)(x)=\begin{cases} \frac{f(r,\omega)}{r}
+w(x,-\tau_k,\delta)(r-\tau_{k-1}), & \tau_{k-1}\leqslant
r\leqslant\tau_k;\\
0, & \text{otherwise};
\end{cases}
\end{equation}
i.e., the $k$-th summand is supported in the layer
$B_{\tau_{k}}(0)\setminus B_{\tau_{k-1}}(0)$.
\smallskip

\noindent{\bf 3.}\,\,\, Assume in addition that $f$ is compactly
supported. Then the summation in $A_{\Pi^\delta}$ is implemented
from $k=1$ to a finite $N$. Represent
$$
Af=\sum\limits_{k=1,\dots,N}\Delta Y_k Af
$$
with the summands
\begin{equation}\label{Eq auxil 4}
\left(\Delta Y^{\tau_k}Af\right)(x)=\begin{cases}
\frac{f(r,\omega)}{r}, & \tau_{k-1}\leqslant
r\leqslant\tau_k;\\
0, & \text{otherwise};
\end{cases}
\end{equation}
supported in the layers $B_{\tau_{k}}(0)\setminus
B_{\tau_{k-1}}(0)$. Comparing (\ref{Eq auxil 3}) with (\ref{Eq
auxil 4}), we have
\begin{align*}
&\left\|\left(A-A_{\Pi^\delta}\right)f\right\|^2=\left\|\sum\limits_{k=1,\dots,N}w(\cdot,-\tau_k,\delta)(|\cdot|-\tau_{k-1})^1_+\,\right\|^2\leqslant\\
&\leqslant\underset{x,k}\sup\,|w(x,-\tau_k,\delta)|^2\sum\limits_{k=1,\dots,N}(\tau_k-\tau_{k-1})^2\,\,\overset{{\rm
\,Lemma}\,\, \ref{L ampl breaks perturb}}\leqslant\\
&\leqslant \,c\,\|f\|^2_{C^2(\Sigma)}\,
\sum\limits_{k=1,\dots,N}(\tau_k-\tau_{k-1})^2\leqslant\,c\,\|f\|^2_{C^2(\Sigma)}\,\delta\,
T\,\underset{\delta\to 0}\to 0,
\end{align*}
where $\tau=T$ is an upper bound of ${\rm supp\,}f$ on $\Sigma$.

Thus, the bounded sequence of the sums $A_{\Pi^\delta}$ converges
to $A$ on a dense set of the smooth compactly supported controls.
This implies (\ref{Eq AI}) for smooth controls.
\smallskip

\noindent{\bf 4.}\,\,\,Approximating an arbitrary $f\in\mathscr F$
by smooth controls and passing to the relevant limit, one extends
(\ref{Eq AI}) to $\mathscr F$.
\end{proof}
\medskip

We denote the limit in (\ref{Eq AI}) by $
\int_{[0,\infty)}dY^\tau\, W\,dX^\tau $ and call this operator the
amplitude integral\,\,\,(AI), meaning that the image $Af$ is
composed from the break amplitudes of the waves $u^{X^\tau f}$ on
their forwards fronts \cite{BKach,B IPI 2022}.
\smallskip

\noindent$\bullet$\,\,\ Recall that $A:\mathscr F\to\mathscr H$ is
a unitary operator. As is easy to show, the AI-representation
\begin{equation}\label{Eq AI^*}
A^*\,=\,A^{-1}\,=\,\int_{[0,\infty)}dX^\tau\,
W^*\,dY^\tau\,:=\,w\,\text{--}\lim\limits_{\delta\to
0}A^*_{\,\Pi^\delta}
\end{equation}
holds, where $A^*$ acts by (\ref{Eq A^*}).

The AI intertwines the projections: the relations
\begin{equation}\label{Eq Intertwine}
AX^\xi\,=\,Y^\xi A,\quad A^*Y^\xi\,=\,X^\xi A^*,\qquad\xi\geqslant
0
\end{equation}
holds, as easily follows from the definitions and/or
orthogonalities (\ref{Eq orthogonal}).

Denote $A^\xi:= A\upharpoonright\mathscr F^\xi$. By (\ref{Eq
Intertwine}), $A^\xi$  is a unitary operator from $\mathscr F^\xi$
to $\mathscr H^\xi$, whereas the AI-representations
\begin{equation}\label{Eq AI repres A^xi A^xi^*}
A^\xi\,:=\,\int_{[\xi,\infty)}dY^\eta\, W\,dX^\eta,\quad
{A^\xi}^{\,\,*}=\int_{[\xi,\infty)}dX^\eta\, W^*\,dY^\eta
\end{equation}
easily follow from (\ref{Eq AI^*}) and orthogonality relations
(\ref{Eq orthogonal}).
\smallskip

\noindent$\bullet$\,\,\ Searching the construction of AI, one can
extend it to $f\in L_2^{\rm loc}(\Sigma)$ as follows. Let $\eta\in
C^\infty(\Sigma)$ obey $0\leqslant \eta(\cdot)\leqslant 1$,
$\eta\!\mid_{0\leqslant\tau\leqslant 1}=1$,
$\eta\!\mid_{\tau\geqslant 2}=0$; denote
$\eta^T:=\eta(\frac{\cdot}{T})$. Then we put
\begin{equation*}
Af\,:=\,\underset{T\to\infty}{\lim}\,A\,(\eta^Tf),
\end{equation*}
where the limit is understood in the sense of the local
$L_2$-convergence. The extended AI acts in the same way:
$$
(Af)(x)\,=\,\frac{f(r,\omega)}{r},\qquad x=r\omega\in\Bbb R^3,
$$
but, sure, the image may not belong to $\mathscr H$.

\subsubsection*{Bilinear forms}

\noindent$\bullet$\,\,\ The following facts are established in
\cite{BV_3,BV_11}.
\smallskip

Fix a $\xi>0$ and denote by $\chi^\xi$ the indicator
(characterictic function) of the part
$\Sigma^\xi=\{(\tau,\omega)\in\Sigma\,|\,\,\tau\geqslant \xi\}$.
Thus, we have $\mathscr F^\xi=\chi^\xi\mathscr F=\{\chi^\xi
f\,|\,\,f\in\mathscr F\}$. Denote $\mathscr
P^\xi:=\chi^\xi\mathscr P$. Recall that $\mathscr U^\xi$ and
$\mathscr D^\xi$ are the reachable and defect subspaces of the
perturbed system: the relation $\mathscr H^\xi=\mathscr
U^\xi\oplus\mathscr D^\xi$ holds, whereas $\mathscr D^\xi$ is
characterized by (\ref{Eq char Def}). Moreover, as is shown in
\cite{BV_11}, the relation
\begin{equation*}
\mathscr D^\xi\,=\,\overline{\{u^p(\cdot,0)\,|\,\,p\in\mathscr
P^\xi\}},\qquad \xi>0
\end{equation*}
is valid, which implies $\mathscr
H^\xi=\overline{\{u^f(\cdot,0)\,|\,\,f\in\mathscr
F^\xi\dotplus\mathscr P^\xi\}}$.

As is shown in \cite{BV_3}, Lemma 2.1, if $f\in\mathscr F^\xi$ and
$Wf=u^f(\cdot,0)=0$ holds then necessarily $f=0$. Nothing is
required to change in the proof to extend this result to the
polynomial controls. So, the map $f\mapsto u^f(\cdot,0)$ from
$\mathscr F^\xi\dotplus\mathscr P^\xi$ to $\mathscr H^\xi$ is {\it
injective} for all $\xi>0$. Note in addition that for $\xi=0$ this
may be wrong: the case ${\rm Ker\,}W\not=\{0\}$ is possible
\cite{BV s-points JMA 2010}.
\smallskip

\noindent$\bullet$\,\,\,The bilinear form
\begin{equation}\label{Eq Form 0}
\langle f,g\rangle_0\,:=\,
(u^f_0(\cdot,0),u^g_0(\cdot,0))_\mathscr H
\end{equation}
is well defined on $\mathscr G^\xi$. If $f,g\in \mathscr F$ holds
then, by the unitarity of $W_0$, we have
$(u^f_0(\cdot,0),u^g_0(\cdot,0))_\mathscr H=(f,g)_\mathscr F$
which follows to
\begin{equation}\label{Eq =}
\langle f,g\rangle_0\,=\,(f,g)_\mathscr F.
\end{equation}

Recall that $R:\mathscr F\to\mathscr F$ is the response operator.
The relations (\ref{Eq u=u0 outside D}) and (\ref{Eq Rf2a=0})
easily follow to the fact that the response $Rf$ is determined by
the values $f\!\mid_{0\leqslant \tau\leqslant 2a}$. For the given
$f,g \in {\mathscr G}^\xi$, represent
\begin{align*}
f=[1-\chi^{2a}]f+\chi^{2a}f\,=:f_1+f_2\,; \qquad
g=[1-\chi^{2a}]g+\chi^{2a}g\,=:g_1+g_2\,,
\end{align*}
where $f_1, g_1$ belong to ${\mathscr F}^\xi$ and vanish for
$\tau>2a$. The perturbed form
\begin{equation}\label{Eq Form perturb}
\langle f,g\rangle\,:=\,\langle f,g\rangle_0 + (Rf_1,g_1)
\end{equation}
is well defined on $\mathscr G^\xi:=\mathscr F^\xi\dotplus\mathscr
P^\xi$. The following results motivates the use of the perturbed
form.
\begin{Lemma}\label{L perturbed C-form}
Let $f,g \in {\mathscr G}^\xi$. The relation
\begin{equation}\label{(fg)GxiC=(WfWg)Hxi}
\langle f,g\rangle\,=\,(u^f({\,\cdot\,},0)\,,\, u^g({\,\cdot\,},0)
)_\mathscr H
\end{equation}
holds for $\xi> 0$.
\end{Lemma}
\begin{proof}\,\,\, Recall that $q\!\mid_{|x|>a}=0$
and begin with the case $\xi<a$.

\noindent{\bf 1.\,\,\,} Take $f,g\in\mathscr G^\xi$. By (\ref{Eq
Delay Rel}), the domain, in which the potential influences on the
waves initiated by the (delayed) controls from $\mathscr F^\xi$,
is $\{(x,t)\,|\,\,~t<|x|-2a+\xi\}$. Outside it one has $u^f=u^f_0$
and $u^g=u^g_0$. In particular, $u^f(x,0)=u^f_0(x,0)$ and
$u^g(x,0)=u^g_0(x,0)$ holds if $|x|\geqslant 2a$, whereas
$f\!\mid_{\tau<2a}=0$ implies $u^f(\cdot,0)\!\mid_{|x|<2a}=0$. By
the aforesaid, if $f\!\mid_{\tau<2a}=0$ holds, then one has
\begin{align}
\notag &
(u^f(\cdot,0),u^g(\cdot,0))_{\mathscr H}=\\
\notag &=\int_{|x|<2a}u^f(x,0)\,u^g(x,0)\,dx+\int_{|x|\geqslant 2a}u^f(x,0)\,u^g(x,0)\,dx=\\
\notag & =\,0\,+\,\int_{|x|\geqslant 2a}u^f(x,0)\,u^g(x,0)\,dx=\int_{|x|\geqslant 2a}u^f_0(x,0)\,u^g_0(x,0)\,dx=\\
\label{Eq 100} &=(u^f_0(\cdot,0), u^g_0(\cdot,0))_{\mathscr
H^\xi}.
\end{align}
\noindent{\bf 2.\,\,\,} For the given $f,g \in {\mathscr G}^\xi$,
represent
\begin{align*}
f=[1-\chi^{2a}]f+\chi^{2a}f\,=:f_1+f_2\,; \qquad
g=[1-\chi^{2a}]g+\chi^{2a}g\,=:g_1+g_2\,,
\end{align*}
and note that $f_1, g_1 \in {\mathscr F}^\xi$. With regard to the
above-made remarks, one has
\begin{align*}
&\left(u^f({\,\cdot\,},0)\,,\, u^g({\,\cdot\,},0)
\right)_{{\mathscr H}^\xi}=\left(u^{f_1+f_2}({\,\cdot\,},0)\,,\,
u^{g_1+g_2}({\,\cdot\,},0) \right)_{{\mathscr H}^\xi}=\\
&=\left(u^{f_1}({\,\cdot\,},0)+u^{f_2}({\,\cdot\,},0)\,,\,
u^{g_1}({\,\cdot\,},0)+u^{g_2}({\,\cdot\,},0)\right)_{{\mathscr H}^\xi}=\\
&= \left(u^{f_1}({\,\cdot\,},0)\,,\, u^{g_1}({\,\cdot\,},0)
\right)_{{\mathscr H}^\xi}+
  \left(u^{f_1}({\,\cdot\,},0)\,,\, u^{g_2}({\,\cdot\,},0) \right)_{{\mathscr H}^\xi}+\\
& +\left(u^{f_2}({\,\cdot\,},0)\,,\, u^{g_1}({\,\cdot\,},0)
\right)_{{\mathscr H}^\xi}+
  \left(u^{f_2}({\,\cdot\,},0)\,,\, u^{g_2}({\,\cdot\,},0) \right)_{{\mathscr H}^\xi}\overset{(\ref{Eq 100})}=\\
& =\left(u^{f_1}({\,\cdot\,},0)\,,\, u^{g_1}({\,\cdot\,},0)
\right)_{{\mathscr H}^\xi}+
  \left(u^{f_1}_0({\,\cdot\,},0)\,,\, u^{g_2}_0({\,\cdot\,},0) \right)_{{\mathscr H}^\xi}+\\
& \left(u^{f_2}_0({\,\cdot\,},0)\,,\, u^{g_1}_0({\,\cdot\,},0)
\right)_{{\mathscr H}^\xi}+
  \left(u^{f_2}_0({\,\cdot\,},0)\,,\, u^{g_2}_0({\,\cdot\,},0)
  \right)_{\mathscr H^\xi}\overset{(\ref{Eq 2.31}),\,(\ref{Eq Form 0})}=\\
&=\left(Cf_1,g_1\right)+\langle f_1,g_2\rangle_0
+\langle f_2,g_1\rangle_0+\langle f_2,g_2\rangle_0\overset{(\ref{Eq C=I+R})}=\\
&=(f_1,g_1)+(Rf_1,g_1)+\langle f_1,g_2\rangle_0 +\langle
f_2,g_1\rangle_0+\langle f_2,g_2\rangle_0
\overset{(\ref{Eq =})}=\\
&=\langle f_1,g_1\rangle_0+(Rf_1,g_1)+\langle f_1,g_2\rangle_0
+\langle f_2,g_1\rangle_0+\langle f_2,g_2\rangle_0=\\
&=\langle f,g\rangle_0+(Rf_1,g_1) \overset{(\ref{Eq Form
perturb})}= \langle f,g\rangle.
\end{align*}

\noindent{\bf 3.\,\,\,} The case $\xi \geqslant a$ is much simpler
and treated in the same way.
\end{proof}

This proof is very close to the proof of Lemma 2 in \cite{BV JIIPP
centrifug pot}.
\smallskip

\noindent$\bullet$\,\,\,Let $0<\xi<a$. Assume that we are given
the operator $R^\xi:=X^\xi R\upharpoonright\mathscr F^\xi$ acting
in $\mathscr F^\xi$. In accordance with (\ref{Eq Rf2a=0}), it
determines the potential radius by
\begin{equation}\label{Eq for a}
a\,=\,\,\inf\left\{b>0\,|\,\,R^\xi f
\!\mid_{\tau>2b-\xi}=0\,\,\,\text{for all}\,\,\,f\in\mathscr
F^\xi\right\}
\end{equation}
and, hence, determines the above used decompositions $f=f_1+f_2$
for $f\in\mathscr G^\xi$. As a consequence, one can write (\ref{Eq
Form perturb}) as
\begin{equation}\label{Eq final perturb form}
\langle f,g\rangle=\langle f,g\rangle_0+(R^\xi
f_1,g_1)\overset{(\ref{(fg)GxiC=(WfWg)Hxi})}=(u^f({\,\cdot\,},0)\,,\,
u^g({\,\cdot\,},0) )_\mathscr H
\end{equation}
and make use of this form in the inverse problem.

\subsubsection*{Space $\tilde{\mathscr H^\xi}$}

\noindent$\bullet$\,\,\,Fix a $\xi>0$ and recall that the map
$f\mapsto u^f(\cdot,0)$ is injective on $\mathscr G^\xi=\mathscr
F^\xi\dotplus\mathscr P^\xi$. In the mean time, by Lemma \ref{L
perturbed C-form} the form $\langle f,g\rangle$ is positive on
$\mathscr G^\xi$. Hence, endowing $\mathscr G^\xi$ with the inner
product $\langle f,g\rangle$, we have a pre-Hilbert space.
Completing it with respect to the corresponding norm, we get a
Hilbert space $\tilde{\mathscr H}^\xi$. We say it to be a {\it
model space}.

By $f\mapsto \tilde u^f(\cdot,0)$  we denote the embedding map
$\mathscr G^\xi\to\tilde{\mathscr H^\xi}$ and call its images
$\tilde u^f(\cdot,0)$ the model waves. In the mean time, the map
$f\mapsto u^f(\cdot,0)$ acts from $\mathscr G^\xi$ to $\mathscr
H^\xi$. By construction, in accordance with
(\ref{(fg)GxiC=(WfWg)Hxi}), the correspondence $U^\xi:\tilde
u^f(\cdot,0)\mapsto u^f(\cdot,0)$\,\,\,($f\in\mathscr G^\xi$) is
an isometry which extends to a unitary operator from
$\tilde{\mathscr H^\xi}$ onto $\mathscr H^\xi$. The model waves
play the role of the isometric copies of the true waves invisible
for the external observer. The observer, which possesses the
response operator, can determine the form $\langle f,g\rangle$,
construct the model space $\tilde{\mathscr H^\xi}$, and find the
copy $\tilde u^f$ of $u^f$ for any $f\in\mathscr G^\xi$.
\smallskip

\noindent$\bullet$\,\,\,The reduced control operator
$W^\xi:=W\upharpoonright\mathscr F^\xi$ acts from $\mathscr F^\xi$
onto $\mathscr U^\xi\subset\mathscr H^\xi$. Its dual $\tilde
W^\xi:f\mapsto \tilde u^f(\cdot,0)$ maps $\mathscr F^\xi$ onto its
image $\tilde{\mathscr U}^\xi\subset\tilde{\mathscr H}^\xi$ under
the embedding. Note the evident relation $U^\xi\tilde
W^\xi=W^\xi$.

The model space $\tilde{\mathscr H^\xi}$ contains a family of
subspaces
$$
\tilde{\mathscr H^\eta}\,:=\,\overline{\{\mathscr
F^\eta\dotplus\mathscr P^\eta\}},\qquad \eta\geqslant\xi
$$
(the closure in $\tilde{\mathscr H^\xi}$ of the images under the
embedding) and the corresponding projections $\tilde Y^\eta$ in
$\tilde{\mathscr H^\xi}$ onto $\tilde{\mathscr H^\eta}$. By the
isometry, we have
\begin{equation}\label{Eq YU=UY}
U^\xi\mathscr H^\eta=\mathscr H^\eta,\quad Y^\eta
U^\xi=U^\xi\tilde Y^\eta,\qquad \eta\geqslant \xi,
\end{equation}
where $Y^\eta$ cuts off functions on $\Bbb R^3\setminus
B_\eta(0)$.

\subsubsection*{Wave visualization and solving IP}

\noindent$\bullet$\,\,\ Fix a $\xi>0$ and recall that the operator
$A^\xi=A\upharpoonright\mathscr F^\xi$ acts from $\mathscr F^\xi$
to $\mathscr H^\xi$. The operator
$V^\xi:={A^\xi}^{\,\,*}W^\xi:\mathscr F^\xi\to\mathscr F^\xi$ is
called a {\it visualizing operator}. It acts by the rule
\begin{align}\label{Eq def V^xi}
\notag & (V^\xi f)(\tau,\omega)
=({A^\xi}^{\,\,*}u^f(\cdot,0))(\tau,\omega)=\\
&=({A}^{*}u^f(\cdot,0))(\tau,\omega)\overset{(\ref{Eq
A^*})}=\tau\,u^f(\tau\omega,0),\quad (\tau,\omega)\in\Sigma.
\end{align}
The external observer, which possesses this operator, gets an
option for a given $f$ to see the "photo" of the invisible wave
$u^f(\cdot,0)$ on the "screen" $\Sigma^\xi$, what motivates the
term "visualization". Let us show how to realize such an option.
Using the unitarity ${U^\xi}^{\,*}{U^\xi}=\Bbb I_{\tilde{\mathscr
H}^\xi}$ and the connection $U^\xi\tilde W^\xi=W^\xi$, we have
\begin{align}
\notag & V^\xi\overset{(\ref{Eq AI repres A^xi A^xi^*})}=
\left[\int_{[\xi,\infty)}dX^\eta\,{W^\xi}^{\,*}\,d\,Y^\eta\right]\,W^\xi
=\\
\notag &\overset{(\ref{Eq
YU=UY})}=\left[\int_{[\xi,\infty)}dX^\eta\,({U^\xi}\tilde
{W^\xi})^{\,*}\,d\,({U^\xi}\tilde Y^\eta {U^\xi}^{\,*}
)\right]\,{U^\xi} \tilde W^\xi=\\
\label{Eq visual}& =\left[\int_{[\xi,\infty)}dX^\eta\, {\tilde
{W^\xi}}^{\,*}\,d\,\tilde Y^\eta\right]\,\tilde W^\xi
\end{align}
and thus represent the defined in (\ref{Eq def V^xi}) operator
$V^\xi$ in terms of the model space. It is a representation that
allows us to solve the inverse problem: everything will be done if
we show how to determine $V^\xi$ from the inverse data.
\smallskip

\noindent$\bullet$\,\,\ The external observer probes the system by
controls $f\in\mathscr F^\xi$ and obtains the operator
$R^\xi:=X^\xi R\upharpoonright\mathscr F^\xi$ as a result of
measurements. Such an information enables him to recover the
potential $q$ in $\Bbb R^3\setminus B_\xi(0)$ by means of the
following procedure.
\smallskip

\noindent{\bf Step 1.\,\,\,}Having $R^\xi$, find the radius of the
potential by (\ref{Eq for a}): this enables one to decompose
controls by $f=f_1+f_2$ . Determine the form $\langle f,g\rangle$
on $\mathscr G^\xi$ by (\ref{Eq final perturb form}). Construct
the model space $\tilde{\mathscr H^\xi}$. Determine the operator
(embedding) $\tilde W^\xi:\mathscr F^\xi\to\tilde{\mathscr H^\xi}$
and its adjoint ${\tilde {W^\xi}}^{\,*}$.
\smallskip

\noindent{\bf Step 2.\,\,\,}Find the projections $\tilde Y^\eta$
in $\tilde{\mathscr H^\xi}$ onto $\tilde W^\xi\mathscr F^\eta$.
Constructing the AI, determine the visualizing operator $V^\xi$ by
(\ref{Eq visual}). Recall that it acts by $(V^\xi
f)(\tau,\omega)=\tau\,u^f(\tau\omega,0)$,\,\,\,\,$(\tau,\omega)\in\Sigma$.
\smallskip

\noindent{\bf Step 3.\,\,\,}Transferring the images $V^\xi f$ from
$\Sigma^\xi$ to $\Bbb R^3\setminus B_\xi(0)$ by the equality
$u^f(x,0)=|x|\,(V^\xi f)(|x|,\frac{x}{|x|})$, recover the operator
$W^\xi=W\upharpoonright\mathscr F^\xi$.
\smallskip

\noindent{\bf Step 4.\,\,\,} Possessing $W^\xi$ and using
$u^{f_{tt}}=u^f_{tt}\overset{(\ref{Eq 1})}=(\Delta-q)u^f$, recover
the graph of the operator $\Delta-q$ by
$$
{\rm graph}\,(\Delta-q)=\{[u^f,u^{f_{tt}}]\,|\,\,f\in\mathscr
F^\xi\cap C^2(\Sigma)\}=\{[W^\xi f,W^\xi
f_{tt}]\,|\,\,f\in\mathscr F^\xi\cap C^2(\Sigma)\}
$$
($[\cdot,\cdot]$ denotes a pair). The graph evidently determines
the potential $q$ in $\Bbb R^3\setminus B_\xi(0)$. The IP is
solved.
\smallskip

Possessing $R^\xi$ for all $\xi>0$, the observer can recover $q$
in the whole $\Bbb R^3$.

\subsubsection*{Comments}

\noindent$\bullet$\,\,\ If the response operator admits the
representation (\ref{Eq Repres R}) then to set up $R^\xi$ is to
give its kernel $p\!\mid_{\tau\geqslant2\xi}$ as the inverse data.
In this case, we have to determine a function $q$ of three
variables from a function $p$ of 1+2+2=5 variables that is an
overdetermined setup of the inverse problem. The question arises
to characterize the kernels, which correspond to the potentials.
One necessary condition is quite traditional and easily seen: $p$
must provide the positivity of the form $\langle f,g\rangle$. Can
one propose a list of the necessary and sufficient conditions? In
a sense, it is a question of the taste and definitions: what a
characterization is. Presumably, a characterization like a rather
long list of conditions in \cite{BV Rendiconti,BKor} can be
proposed. The meaning of these conditions is to provide
realizability of the procedures of the type {\it Step 1--Step 4}.
However, in our opinion, a simple characterization, like in
one-dimensional problems, is hardly possible.

On the not over-determined setup of the scattering problems see,
e.g., \cite{Rakessh}.
\smallskip

\noindent$\bullet$\,\,\ The model space $\tilde{\mathscr H^\xi}$
is a rather specific object: it is not a function space, since its
elements cannot be assigned to certain subsets (supports) in
$\Sigma^\xi$. This situation is not new: the same thing occurs in
problems in the bounded domains \cite{ABI}. Such effects are
connected with the quality of controllability of the system: the
presence of approximate controllability, but the absence of exact
controllability.

Nevertheless, such an exotic object can be adapted for the
elaboration of numerical algorithms. The thing is that
$\tilde{\mathscr H^\xi}$ is in fact an intermediate object,
whereas in algorithms the Amplitude Integral is in the use. Its
version (the so-called {\it amplitude formula}), which is the
result of the differentiation of the AI (\ref{Eq AI repres A^xi
A^xi^*}) w.r.t. $\xi$, is quite suitable for numerical realization
\cite{B IP 97,BGot,BIKS,Pestov Film}.

\noindent{\bf Key words:}\,\, three-dimensional dynamical system
governed by the locally perturbed wave equation, determination of
potential from inverse scattering data, boundary control method.
\bigskip

\noindent{\bf MSC:}\,\,35R30, 35Lxx, 47Axx.
\bigskip

\end{document}